\documentclass[12pt,a4paper]{article}
\usepackage{dsfont}
\usepackage{mathtext}
\usepackage[T2A]{fontenc}
\usepackage[utf8]{inputenc} 
\usepackage[russian,english]{babel} 
\usepackage{amsthm,amssymb,amsmath,amsfonts}
\usepackage{mathtools}
\usepackage{enumerate}
\usepackage[indentfirst]{titlesec}
\usepackage[all]{xy}

\newtheorem{Th}{Theorem}[section]
\newtheorem{Prop}[Th]{Proposition}
\newtheorem{Lemma}[Th]{Lemma}

\newtheorem{Cond}[Th]{Condition}
\newtheorem{Quest}{Question}
\newtheorem*{Th*}{Theorem}

\theoremstyle{definition} %%%%
\newtheorem{Def}[Th]{Definition}
\newtheorem{Remark}[Th]{Remark}
\newtheorem{Example}[Th]{Example}

\newcommand{\Sing}{\operatorname{Sing}}
\newcommand{\Supp}{\operatorname{Supp}}

\newcommand{\MP}{\mathds{P}}
\newcommand{\MQ}{\mathds{Q}}

\newcommand{\MZ}{\mathds{Z}}

\title{Rationally connected non Fano type varieties}
\date{}
\author{Igor Krylov\footnote{
This is a pre-print of an article published in European Journal of Mathematics. The final version is available online at https://doi.org/10.1007/s40879-017-0201-1 }}
\begin{document}

\maketitle

\begin{abstract}
Varieties of Fano type are very well behaved with respect to the MMP, and they are known to be rationally connected. We study a relation between the classes of rationally connected varieties and varieties of Fano type. It is known that these classes are birationally equivalent in dimension 2. We give examples of rationally connected varieties of dimension $\geqslant 3$ which are not birational to varieties of Fano type, thereby answering the question of Cascini and Gongyo \cite[Question~5.2]{CG}.
\end{abstract}

%\noindent

%%%%%%%%%%%%%%%%%%%%%%%%%%%%%%%%%%%%%%%%%%%%%%%%%%%%%%%%%%%
%%%%%%%%%%%%%%%%%%%%%%%%%%%%%%%%%%%%%%%%%%%%%%%%%%%%%%%%%%%
%%%%%%%%%%%%%%%%%   Introduction     %%%%%%%%%%%%%%%%%%%%%%
%%%%%%%%%%%%%%%%%%%%%%%%%%%%%%%%%%%%%%%%%%%%%%%%%%%%%%%%%%%
%%%%%%%%%%%%%%%%%%%%%%%%%%%%%%%%%%%%%%%%%%%%%%%%%%%%%%%%%%%

\section{Introduction}

The log minimal model program (MMP) is one of the key notions in birational geometry. Finding when can we run the MMP for
a pair $(X,D)$ ($D$-MMP) is one of the central subjects and is still being developed. If $X$ is a variety of Fano type then we can run the $D$-MMP on it for any divisor $D$ on $X$ \cite[Corollary~2.7]{PS}. 
We say that a normal projective variety $X$ is of Fano type if there is an effective $\mathds{Q}$-divisor $H$ on $X$ such that the pair $(X,H)$ is Kawamata log terminal and $-(K_X+H)$ is ample. Varieties of Fano type have been introduced by Shokurov and Prokhorov in \cite{PS}. The Fano type property is preserved under flips and contractions (\cite[Lemma~2.8]{PS}). Thus if we run the MMP on a variety of Fano type, then on any stage we have a variety Fano type. In particular the result of running the MMP on a variety of Fano type is a variety of Fano type. Varieties of Fano type appear naturally as quotients of Fano varieties by a finite group and as resolutions of some singular Fano varieties.

It would be nice if we could say if in a given birational class there is a variety which behaves well with respect to the $D$-MMP for any divisor $D$. For example a variety of Fano type. On the other hand it is known that varieties of Fano type are rationally connected \cite[Theorem~1]{Zh}. Thus it is natural to ask if the converse is true. Even in dimension $2$ it is not: a blow up of $10$ points on $\mathds{P}^2$ in general position is not a variety of Fano type. We may broaden the converse statement, however.

\begin{Quest}[{\cite[Question~5.2]{CG}}] \label{Q1}
Let $X$ be a rationally connected variety. Is $X$ birationally equivalent to a variety of Fano type?
\end{Quest}

In dimension $2$ the answer is positive since every rationally connected surface is rational. In this paper we construct examples for which the answer is negative in dimension $\geqslant3$. Namely the purpose is to prove the following two theorems.

\begin{Th} \label{TechTh}
\begin{enumerate}[(i)]
	\item Let $W$ be a generic smooth divisor of degree $(2M,2l)$ on $\mathds{P}^M\times \mathds{P}^1$, $M\geqslant3$, $l\geqslant3$. Let $\sigma:V\to \mathds{P}^m\times\mathds{P}^1$ be the double cover branched over $W$. Then every birational map from $V$ to a variety admitting a Mori fiber space is an isomorphism.
	\item Let $W=(\mathds{C}^6\setminus Z\big(\langle u,v\rangle\cap\langle x,y,z,w\rangle\big)/(\mathds{C}^*)^2$, where $(\mathds{C}^*)^2$-action is given by the matrix 

\[\left(\begin{array}{cccccc}
u&v&x&y&z&w\\
0&0&1&1&2&3\\
1&1&-3&-3&0&0\end{array}\right)\]

and let $X\subset W$ be the hypersurface of bi-degree $(6,0)$ given by the equation 
\begin{align*}
Q=w^2+z^3+(u^{12}+v^{12})M_4(x,y)z+R_{18}(u,v)x^2y^2(x-y)^2,
\end{align*}
where $M_4$ and  $R_{18}$ are generic homogeneous polynomials of degrees $4$ and $18$ respectively. Then every birational map from $X$ to a variety admitting a Mori fiber space is an isomorphism.
\end{enumerate}

\end{Th}

\begin{Th} \label{MainTh}
The varieties $V$ and $X$, described in Theorem \ref{TechTh}, are not birationally equivalent to a variety of Fano type.
\end{Th}

Since $V$ and $X$ are rationally connected (Lemma \ref{RC}) we conclude that the answer to Question \ref{Q1} is negative for these varieties.

To prove Theorem \ref{TechTh} we use techniques of birational rigidity. We say that a Fano fibration $\pi:U\to \mathds{P}^1$ is \emph{birationally superrigid} if any birational map $\chi:U\dasharrow Y$ to a variety, admitting a Mori fiber space $Y\to Z$, is a \emph{fiberpreserving map}. That is the following diagram is commutative
\begin{align*}
\xymatrix
{ 
	U\ar@{-->}[r]^\chi \ar[d]_\pi & \ar[d] Y \\
	\mathds{P}^1\ar@{=}[r] & Z
}
\end{align*}
and $\chi$ is an isomorphism on a generic fiber of $\pi$. In particular it means that $U$ is not birational to a Fano variety with Picard rank $1$ or a conic bundle and hence is not rational. 

The birational superrigidity is usually proven using Noether-Fano inequality, its origin is the theorem on generation of the Cremona group $\operatorname{Cr}_2=\operatorname{Bir}(\mathds{P}^2)$ by $\operatorname{PGL}_3(\mathds{C})$ and the standard quadratic transformation. Let $U$ and $U^\prime$ be varieties which admit a Mori fiber space and suppose $\chi:U \dasharrow U^\prime$ is a birational map which is not an isomorphism. Then by Noether-Fano inequality (Proposition 2.8) there is a very singular $\mathds{Q}$-linear system $\mathcal{M}$ on $U$ which is $\mathds{Q}$-linearly equivalent to $-K_U$ on fibers of $\pi$.

The variety $V$ admits a Fano fibration  $\pi$ over $\mathds{P}^1$, the map $\pi$ is the composition of the double cover $\sigma$ and the projection onto $\mathds{P}^1$. Indeed, the restriction of $\sigma$ to a fiber of $\pi$ is a double cover of $\mathds{P}^M$ branched over a hypersurface of degree $2M$, thus the fiber of $\pi$ is a Fano variety. We use the techniques developed in \cite{Pukh_Dir_Prod} and \cite{Pukh_Sing_Fano} to prove $(i)$ of Theorem \ref{TechTh}: we impose generality conditions on the branching divisor which help us to control the singularities of linear systems.

The variety $X$ admits a del Pezzo fibration of degree $1$. Indeed, the variety $W$ in $(ii)$ of Theorem \ref{TechTh} admits a $\mathds{P}(1,1,2,3)$-fibration over the projective line. Since $X$ is a sextic in every fiber, this fibration restricted to $X$ is a del Pezzo fibration of degree $1$. A smooth variety admitting a del Pezzo fibration of degree $1$ is birationally superrigid if it satisfies the famous $K^2$-condition, that is $K^2$ is not in the interior of Mori cone \cite[Theorem~2.1]{Pukh123}. Unfortunately the superrigidity is not enough for us as it allows birational maps which are not isomorphism to other del Pezzo fibrations of degree $1$: fiberpreserving maps. Also $X$ has $36$ ordinary double points (Lemma \ref{SingX}), thus \cite[Theorem~2.1]{Pukh123} is not applicable. Furthermore a priori $X$ may not be $\mathds{Q}$-factorial. We use Lefschetz-type result (\cite[Theorem~4.1]{Rams}) to prove that $X$ admits a Mori fiber space over $\mathds{P}^1$. 

This allows us to use Noether-Fano inequality. We use the results from \cite{Pukh123} and \cite{JPDP} to show that the linear systems may be sufficiently singular only at the cusps of curves of the anticanonical degree $1$ in a fiber of del Pezzo fibration. Then we prove that all such points are the ordinary double points of $X$ or of the fibers. We show that the pair $(X,\mathcal{M})$ is not too singular at these points. At last we conclude that all the birational maps from $X$ to varieties admitting a Mori fiber space are isomorphisms.

To prove Theorem \ref{MainTh}, we show that any variety of Fano type is birationally equivalent to a variety with big anticanonical divisor admitting a Mori fiber space (Proposition \ref{MMP}). On the other hand $-K_V$ and $-K_X$ are not big (Lemma \ref{KBig4} and Lemma \ref{KBig3}).

\subsection*{Acknowledgments}
The author would like to thank John Ottem and Ivan Cheltsov, for bringing this question to his attention and for many helpful conversations, and Antony Manioca, Milena Hering, Johan Martens, and Michael Wemyss for useful suggestions. The author also expresses his gratitude to the referee for pointing out many ways to improve this paper. The author was supported by an Edinburgh PCDS scholarship. 

%%%%%%%%%%%%%%%%%%%%%%%%%%%%%%%%%%%%%%%%%%%%%%%%%%%%%%%%%%%
%%%%%%%%%%%%%%%%%%%%%%%%%%%%%%%%%%%%%%%%%%%%%%%%%%%%%%%%%%%
%%%%%%%%%%%%%%%%%   Preliminaries    %%%%%%%%%%%%%%%%%%%%%%
%%%%%%%%%%%%%%%%%%%%%%%%%%%%%%%%%%%%%%%%%%%%%%%%%%%%%%%%%%%
%%%%%%%%%%%%%%%%%%%%%%%%%%%%%%%%%%%%%%%%%%%%%%%%%%%%%%%%%%%

\section{Preliminaries}
All the varieties in this paper are considered to be normal, projective, and defined over $\mathds{C}$ unless stated otherwise. 
We write $\sim$ ($\sim_\mathds{Q}$) for linear ($\mathds{Q}$-linear) equivalence of divisors ($\mathds{Q}$-divisors). All Mori fiber spaces are assumed to be in the Mori category, that is terminal and $\MQ$-factorial. When we say that $\pi: X \to \MP^1$ is a Fano (del Pezzo) fibration we assume that it is a Mori fiber space with a generic fiber being a Fano variety (del Pezzo surface).

\begin{Def}[{\cite[p.~6]{Pr_C}}]
Let $D$ be a $\mathds{Q}$-divisor on a variety $X$ such that $K_X+D$ is $\mathds{Q}$-Cartier. 
Let $\sigma:\widetilde{X} \to X$ be a birational morphism and let $\widetilde{D}=\sigma^{-1}(D)$ be the proper transform of $D$. 
Then we can write
\begin{align*}
K_{\widetilde{X}}+\widetilde{D}\sim_\mathds{Q}\sigma^*(K_X+D)+\sum_E a(E,X,D)E,
\end{align*}
where $E$ runs through all the distinct exceptional divisors of $\sigma$ on $\widetilde{X}$ and $a(E,X,D)$ is a rational number. The number $a(E,X,D)$ is called the \emph{discrepancy} of a divisor $E$ with respect to the pair $(X,D)$.

We say that the pair $(X,D)$ is \emph{terminal} (resp. \emph{canonical}, \emph{log terminal}, \emph{log canonical}) \emph{at a subvariety} $Z$ of $\operatorname{codim}_X Z\geqslant 2$ if for every birational morphism $\sigma$ to $X$ the inequality $a(E,X,D)>0$ (resp. $a(E,X,D)\geqslant0$, $a(E,X,D)>-1$, $a(E,X,D)\geqslant-1$) holds for every prime $\sigma$-exceptional divisor $E$ such that $\sigma(E)=Z$.

Let $\mathcal{M}$ be a linear system on $X$ and let $\lambda\in \mathds{Q}_{>0}$, we say that the pair $(X,\lambda \mathcal{M})$ is \emph{terminal} (resp. \emph{canonical}, \emph{purely log terminal}, \emph{log canonical}) if for every subvariety $Z\subset X$ of $\operatorname{codim}_X Z\geqslant 2$ the pair $(X,\lambda D)$ is terminal (resp. canonical, log terminal, log canonical) at $Z$ for a generic divisor $D\in \mathcal{M}$. Note that the system $\mathcal{M}$ may consist of only one divisor $D$, in this case we say that $(X,\lambda D)$ is \emph{terminal} (resp. \emph{canonical}, \emph{purely log terminal}, \emph{log canonical}). If $D=0$, we simply say that $X$ has terminal (resp. canonical, log terminal, log canonical) singularities. 

Let $D=\sum \alpha_i D_i$ be a divisor on a variety $X$. We say that the pair $(X,D)$ is $\emph{klt}$ if it is purely log terminal and $\alpha_i<1$ for all $i$.
\end{Def}

\begin{Example}
\begin{itemize}

	\item Consider the pair $(S,C)$, where $S$ is a smooth surface and $C\subset S$ is a smooth curve. Then the pair is canonical at every point of $C$ and is terminal at every other point.

	\item Consider the pair $(\mathds{P}^2, \big| L \big|)$, where $L$ is a line. Then the pair is terminal at every point $P\in \mathds{P}^2$ since a generic line does not pass through $P$.

	\item Let $\mathcal{L}$ be the linear system of lines on $\mathds{P}^2$ passing through a point $P$. Then the pair $(\mathds{P}^2,\mathcal{L})$ is canonical at $P$ and is terminal elsewhere.
\end{itemize}
\end{Example}

\begin{Def}[{\cite[Lemma-Definition~2.6]{PS}}]
We say that a variety $X$ is of \emph{Fano type} if there exists an effective $\mathds{Q}$-divisor $D$ on $X$ such that $-(K_X+D)$ is $\mathds{Q}$-Cartier and ample, and the pair $(X,D)$ is klt.
\end{Def}

\begin{Example}
\begin{enumerate}[(i)]
	\item 
The quotient $Y=X/G$ of a Fano variety by a finite group is a variety of Fano type. Indeed, since $K_X=f^*\big (K_Y+\frac{R}{\mid G \mid}\big )$, where $R$ is the ramification divisor, the divisor $-\big (K_Y+\frac{R}{\mid G \mid}\big )$ is ample. The pair $\big (Y,\frac{R}{\mid G \mid}\big )$ is klt by \cite[Proposition~3.16]{KSP}, hence $Y$ is a variety of Fano type. 
	\item Suppose $(S,D_S)$ is a log del Pezzo surface (that is $-(K_S+D_S)$ is ample) with klt singularities. Naturally it is of Fano type. 
Let $f:X\to S$ be the minimal resolution of singularities of $S$ and let $D$ be the proper transform of $D_S$. Then $X$ is also of Fano type. 
Indeed we can write 
$$K_X+D=f^*(K_S+D_S)+\sum a_i E_i,$$
where $-1< a_i\leqslant 0$, thus $F=-(K_X+D-\sum a_i E_i)$ is an effective divisor which is the pullback of an ample divisor. 
Note that $F$ is not ample because $F\cdot E_i=0$. 
Consider the pair $\big(X,D - \sum (1-\varepsilon_i)a_i E_i\big)$, for some nonnegative rational numbers $\varepsilon_i\ll 1$. Then the pair is klt and  $-K_X+D-\sum (1-\varepsilon_i)a_i E_i$ is ample.
	\item Suppose $X$ is a variety of Fano type. Then there is a divisor $D$ such that $-(K_X+D)$ is ample. Thus $-K_X$ is a sum of an ample and an effective divisor, therefore it is big. Let $\sigma:S\to \mathds{P}^2$ be a blow up of $10$ points in general position then $-K_S$ is not big, therefore $S$ is not of Fano type. Note that if the points are not in a general position then $-K_S$ might still be big \cite{BRS}.
	%\item Varieties of Fano type also appear as exceptional divisors of some extremal contractions.
\end{enumerate}
\end{Example}

\begin{Th} [Inversion~of~adjunction,~{\cite{AdjK}}] \label{ThAdj}
Let $(X,S+B)$ be a log pair such that $S$ is a reduced divisor which has no common component with the support of $B$, let $S^\nu$ denote the normalization of $S$, and let $B^\nu$ denote the different of $B$ on $S^\nu$.
Then $(X, S+B)$ is log canonical near $S$ if and only if $(S^\nu, B^\nu)$ is log canonical.
Let $S\subset X$ be an irreducible divisor and let $D$ be an effective $\mathds{Q}$-Cartier divisor. Assume that $K_X+S$ is $\mathds{Q}$-Cartier and that the pair $(X,S)$ is purely log terminal. Then the pair $(X,S+D)$ is log canonical in a neighborhood of $S$ if and only if the pair $\big(S,\operatorname{Diff}(D)\big)$ is log canonical. 
\end{Th}

\begin{Remark} \label{Adj}
For the definition of the \emph{different} we refer to \cite[Chapter~16]{K}. If $X$ is smooth in codimension $2$, then $\operatorname{Diff}(D) = D\big\vert_S$ by \cite[Corollary~16.7]{K} (case 16.6.3, $m=1$).

We use Inversion of adjunction to study singularities of pairs as follows. 
Let $F$ be a prime normal divisor on a variety $X$ and let $Z \subset F$ be a subvariety. 
Let $D$ be an effective divisor such that $(X,D)$ is not canonical at some $Z\subset F$. 
Suppose $X$ is smooth along $Z$ then $(X,D+F)$ is not log canonical at $Z$ and $X$ is smooth in codimension $2$ in a neighborhood of $Z$.
Thus by Theorem \ref{ThAdj} the pair $(F,D\big\vert_F)$ is not log canonical at $Z$. 
\end{Remark}

\begin{Def}
Let $\chi:V\dasharrow \bar{V}$ be a birational map between varieties admitting Fano fibrations over the projective line. We say it is \emph{fiberpreserving} (with respect to these fibrations) if there is a commutative diagram
\begin{displaymath}
\xymatrix
{ 
	V\ar@{-->}[r]^\chi \ar[d]_\pi & \ar[d]^{\bar{\pi}} \bar{V}\\
	\mathds{P}^1\ar@{=}[r] & \mathds{P}^1,
}
\end{displaymath}
and $\chi$ is an isomorphism on a generic fiber of $\pi$. 
\end{Def}

Consider a $\mathds{P}^1$-bundle $\pi:S\to\mathds{P}^1$. We can blow up a point on a fiber $F$ of $\pi$ and then contract the proper transform of $F$. This is an elementary transformation of $\mathds{P}^1$-bundles and this is the simplest example of a fiberpreserving map.

The following theorem lets us prove that there are no fiberpreserving maps between Fano fibrations.

\begin{Th}[{\cite[Theorem~1.5]{CC}}] \label{ThV}
Let $Z$ be a smooth curve. Suppose that there is a commutative diagram
\begin{displaymath}
\xymatrix
{ 
	V\ar@{-->}[r]^\rho \ar[d]^\pi & \ar[d]_{\bar{\pi}} \bar{V} \\
	Z\ar@{=}[r] & Z
}
\end{displaymath}
such that $\pi$ and $\bar{\pi}$ are flat morphisms, and $\rho$ is a birational map that induces an isomorphism
\begin{align*}
\rho\big\vert_{V\setminus X}:V\setminus X\to \bar{V}\setminus\bar{X}
\end{align*}
where $X$ and $\bar{X}$ are scheme fibers of $\pi$ and $\bar{\pi}$ over a point $O\in Z$, respectively. Suppose that the varieties $V$ and $\bar{V}$ have terminal $\mathds{Q}$-factorial singularities, the divisors $-K_V$ and $-K_{\bar{V}}$ are $\pi$-ample and $\bar{\pi}$-ample respectively, the fibers $X$ and $\bar{X}$ are irreducible, and the variety $X$ has log terminal singularities. Suppose also that the pair $(V,\lambda \mathcal{M})$ is canonical at subvarieties of $X$ for every mobile linear system $\mathcal{M}$ and rational number $\lambda$ such that $K_V+\lambda\mathcal{M}\sim_\mathds{Q} \pi^*(H)$ for some divisor $H$ on $Z$. Then $\rho$ is an isomorphism.
\end{Th}

The proof of Theorem \ref{ThV} repeats the proof of \cite[Theorem~1.5]{CC} except that we do not need Inversion of adjunction. Essentially Cheltsov has proven Theorem \ref{ThV} and then applied Inversion of adjunction to get \cite[Theorem~1.5]{CC}.

Let $X$ and $Y$ be varieties admitting a Mori fiber space. Let $\chi:X\dasharrow Y$ be a birational map which is not an isomorphism. Then, under some conditions on $K_X$, there is a linear system $\mathcal{M}$ on $X$ such that if it is scaled to $-K_X$ with the number $\lambda$, the pair $(X,\lambda \mathcal{M})$ is not canonical. These kind of statements are called Noether-Fano inequalities, they originate from the study of Cremona group by Noether and from the works of Fano. For examples of such statements see the proofs of \cite[Theorem~1.4.1]{Cheltsov-Fano} and \cite[Proposition~2.1]{Pukh123} or the theorem above.

I formulate this condition in the following way.

\begin{Prop}[Noether-Fano inequality] \label{NF}
Let $\pi: X\to \mathds{P}^1$ be a Fano fibration such that $-K_X$ is not big. Suppose that there is a birational map $\chi:X\dasharrow Y$ which is not an isomorphism and suppose $Y$ admits a Mori fiber space. Then there is a mobile linear system $\mathcal{M}$ on $X$ and numbers $s, \lambda \in \mathds{Q}_{\geqslant 0}$ such that $K_X+\lambda\mathcal{M}\sim_\mathds{Q} sF$, where $F$ is a fiber of $\pi$, and the pair $(X,\lambda\mathcal{M})$ is not canonical.
\end{Prop}
\begin{proof}
Let $\mathcal{M}_Y$ be a base point free linear system on $Y$ and let $\mathcal{M}$ be its proper transform on $X$. Let 
%$W$, $\varphi: W\to X$, and $\psi:W\to Y$ be a resolution of the birational map $\chi$. 
\begin{displaymath}
\xymatrix{
  & W \ar[dl]_{\varphi} \ar[dr]^{\psi} & \\
X \ar@{--}[r] &    \text{--}^\chi   \ar@{-->}[r] & Y %not long enough
}
\end{displaymath}
be a resolution of the birational map $\chi$. 
Let $\mathcal{M}_W=\psi^{*}(\mathcal{M}_Y)$ then for any $\lambda\in\mathds{Q}$
\begin{align}
\varphi^*(K_X+\lambda \mathcal{M}) + \sum a_i E_i \sim_\mathds{Q}  K_W + \lambda\mathcal{M}_W 
\sim_\mathds{Q} \psi^*(K_Y+\lambda \mathcal{M}_Y) + \sum b_j E_j,
\end{align}
where $E_i$ and $E_j$ are the exceptional divisors of $\varphi$ and $\psi$ respectively and $a_i$ and $b_j$ are the discrepancies.
Note that $Y$ admits a Mori fiber space, therefore $Y$ is terminal, in particular $b_j>0$.

Suppose there exists a Mori fiber space $\pi_Y:Y\to Z$ with $\operatorname{dim} Z>0$. 
Let $H$ be a very ample divisor on $Z$ and let $\mathcal{M}_Y=\big| \pi_Y^*(H) \big|$. 
Then there are rational numbers $n\geqslant 0$ and $l$ such that $\mathcal{M}\subset \big| -nK_X +lF \big|$. 

If $n=0$, then $Z\cong\mathds{P}^1$ and $\chi$ is fiberpreserving. 
Hence by Theorem \ref{ThV} there exists a mobile linear system $\mathcal{L}$ such that $K_X + \lambda\mathcal{L} \sim_\mathds{Q} \pi^*(\Gamma)$, where $\Gamma$ is a divisor on $Z$ and $(X,\lambda\mathcal{L})$ is not canonical. 
Clearly, $\pi^*(\Gamma)\sim_{\mathds{Q}}sF$ for some $s\in \mathds{Q}$. We claim that $s\geqslant 0$. 
Indeed, $-K_X$ is not big, hence not in the interior of the closure of the cone of mobile divisors. 
On the other hand it is
\begin{align*}
-K_X \sim_{\mathds{Q}} \lambda \mathcal{L} - s F.
\end{align*}
The linear systems $\mathcal{L}$ and $F$ are mobile and for $D \in \mathcal{L}$ we have $D\not\sim_\mathds{Q} kF$ for any $k\in \mathds{Q}$. 
Hence if $s<0$, then $\lambda\mathcal{L}-sF$ is in the interior of the closure of the cone of mobile divisors, contradiction. 
Thus $\mathcal{L}$, $\lambda$, and $s$ are the required linear system and rational numbers.

Now suppose $n>0$. 
Set $\lambda=\frac{1}{n}$ then $K_X+\lambda \mathcal{M}\sim_\mathds{Q} \frac{l}{n}F$, and $l\geqslant0$ since $-K_X$ is not big. 
Let $C$ be a curve in a general fiber of $\pi_Y$, then $C\cdot D=0$ for $D\in \mathcal{M}_Y$. 
We have $C\cdot K_Y<0$ by definition of a Mori fiber space and $\psi^{-1}(C)\cdot E_j=0$ for all $\psi$-exceptional divisors $E_j$ since $C$ is contained in a general fiber.
Let us intersect (1) with $\psi^{-1}(C)$, then
\begin{align*}
\frac{l}{n} \psi^{-1}(C)\cdot \varphi^* (F) + \sum a_i \psi^{-1}(C) \cdot E_i  = C\cdot K_Y <0.
\end{align*}
The cycle $\psi^{-1}(C)$ is numerically effective since $C$ is numerically effective.
Therefore we have $\psi^{-1}(C) \cdot \varphi^*(F)\geqslant 0$ and $\psi^{-1}(C)\cdot E_i\geqslant 0$. 
Thus $a_i<0$ for some $i$.

Suppose $\operatorname{dim}Z=0$, that is $Y$ is a Fano variety with $\operatorname{Pic} Y\cong \mathds{Z}$.
Let $\mathcal{M}_Y$ be a very ample linear system on $Y$. 
Let $\lambda$ and $s$ be rational numbers such that $K_X+\lambda\mathcal{M}\sim_\mathds{Q} sF$, where $F$ is a fiber of $\pi$.
Such numbers exists since otherwise we have $\mathcal{M} \subset l F$ which implies that $Z \cong \MP^1$.
We also have $s \geqslant 0$ because $-K_X$ is not big. 
Thus we may rewrite (1) using $K_X+\lambda\mathcal{M} \sim_\mathds{Q} sF$
\begin{align*}
\psi^*(K_Y+\lambda\mathcal{M}_Y)\sim_\mathds{Q}\sum (a_j-b_j)E_j+\sum a_i E_i +\varphi^*(sF),
\end{align*}
where the first sum is over all $\psi$-exceptional divisors, the second sum is over all $\varphi$-exceptional divisors which are not $\psi$-exceptional, and we set $a_j=0$ if $E_j$ is not $\varphi$-exceptional. 
Note that the coefficients in the first sum are $(a_j-b_j)$ since some $E_j$ are exceptional for both $\psi$ and $\varphi$.
Suppose $(X,\lambda\mathcal{M})$ is canonical, then $a_i\geqslant 0$ for all $i$ and therefore $\sum a_i E_i + \varphi^*(sF)$ is effective. 
Thus we may apply negativity lemma \cite[Lemma~2.19]{K} to conclude that $a_j\geqslant b_j > 0$.
This means that all $\psi$-exceptional divisors are also $\varphi$-exceptional.
Suppose $\varphi$ has $K$ exceptional divisors and $\psi$ has $L$ exceptional divisors, then we have shown that $K\geqslant L$.
We compute $\operatorname{rk}\operatorname{Pic} W$ in two different ways to arrive to a contradiction
\begin{align*}
2+L\leqslant 2+K=\operatorname{rk}\operatorname{Pic} X + K = \operatorname{rk}\operatorname{Pic} W = \operatorname{rk}\operatorname{Pic} Y + L = 1 + L.
\end{align*}
\end{proof}

\begin{Remark}
We can see from the proof that Proposition \ref{NF} holds under weaker assumptions. 
We say that $X$ satisfies $K$\emph{-condition} if $-K_X$ is not in the interior of the closure of the cone of mobile divisors. 
If $-K_X$ is not big, that is, $-K_X$ is not in the interior of the closure of the cone of effective divisors, then $X$ satisfies $K$-condition.
The converse does not hold.
The $K$-condition typically appears when one studies birational rigidity of Fano fibrations.
In Proposition \ref{NF} we may replace the requirement for $-K_X$ to not be big with the $K$-condition.
\end{Remark}

%%%%%%%%%%%%%%%%%%%%%%%%%%%%%%%%%%%%%%%%%%%%%%%%%%%%%%%%%%%
%%%%%%%%%%%%%%%%%%%%%%%%%%%%%%%%%%%%%%%%%%%%%%%%%%%%%%%%%%%
%%%%%%%%%%%%%%%%%    Dimension 4     %%%%%%%%%%%%%%%%%%%%%%
%%%%%%%%%%%%%%%%%%%%%%%%%%%%%%%%%%%%%%%%%%%%%%%%%%%%%%%%%%%
%%%%%%%%%%%%%%%%%%%%%%%%%%%%%%%%%%%%%%%%%%%%%%%%%%%%%%%%%%%

\section {Example in dimension $\geqslant4$}
In this section we construct examples in dimension $\geqslant4$ and prove Theorem \ref{TechTh}, (i).

Let $W$ be a generic smooth divisor of bi-degree $(2M,2l)$ on $\mathds{P}^M\times \mathds{P}^1$, $M\geqslant3$, $l\geqslant3$. Let $\sigma:V\to \mathds{P}^M\times \mathds{P}^1$ be the double cover branched over $W$. The variety $V$ is smooth because $W$ is smooth and $\operatorname{Pic}(V)=\mathds{Z}\oplus\mathds{Z}$ by Lefschetz hyperplane section theorem. Let $\pi:V\to\mathds{P}^1$ be the composition of $g$ and the projection onto $\mathds{P}^1$. The restriction of $\sigma$ to any fiber $F$ of $\pi$ is the double cover of $\mathds{P}^M$ branched over a hypersurface of degree $2M$, therefore $F$ is a Fano variety. Thus $\pi:V\to \mathds{P}^1$ is a Mori fiber space.

The superrigidity of the variety $V$ has been considered in \cite[Theorem~1]{Pukh04}. It was shown that if $V$ satisfies the generality conditions from \cite{Pukh04}, then for every variety $Y$ admitting a Mori fiber space $\bar{\pi}:Y\to Z$ and birational to $V$ and every birational map $\chi: V\dasharrow Y$, we must have $Z=\mathds{P}^1$ and $\chi$ is fiberpreserving with respect to $\pi$ and $\bar{\pi}$. But we need more than this as we do not know how do fiberpreserving maps affect the canonical class.

We now describe some sufficient conditions on the branching divisor (modification of those in \cite[p.~22-23]{Pukh_Dir_Prod}) which let us control the singularities of linear systems on double covers of $\mathds{P}^M$. Then we show that the locus of hypersurfaces of degree $2M$ not satisfying these conditions is of codimension at least $2$. As a generic curve in the space of hypersurfaces of degree $2M$ does not pass through this locus, every fiber of $\pi$ satisfies these conditions.

Let $W_X\subset\mathds{P}^M$ be a hypersurface of degree $2M$, $M\geqslant3$. For a nonsingular point $x\in W_X$ fix a system of affine coordinates $z_1,\dots,z_M$ on $\mathds{P}^M$ with the origin at $x$ and set
\begin{align*}
q_1+q_2+\dots+q_{2M}=0
\end{align*}
to be the equation of the hypersurface $W_X$ in these affine coordinates, where $q_i=q_i(z)$ are homogeneous polynomials of degree $\deg q_i=i$. We can assume that $q_1=z_1$ since $W_X$ is smooth at $x$. Denote 

\begin{align*}
\bar{q_i}=\bar{q_i}(z_2,\dots,z_M)=q_i\big\vert_{\{z_1=0\}}=q_i(0,z_2,\dots ,z_M).
\end{align*}

\begin{Cond} \label{Regularity}
If $M\geqslant 5$, then $W_X$ satisfies the condition at a smooth point $x$ if the rank of the quadratic form $\bar{q_2}$ is at least $2$.

For $M=4$ we need either
\begin{enumerate}[(i)]
	\item The rank of the quadratic form form $\bar{q_2}$ is at least $2$
	\item or the rank of $\bar{q_2}$ is $1$ and the following additional condition is satisfied. Without loss of generality we assume that $\bar{q_2}=z_2^2$.
We require that one of the polynomials $\bar{q}_3(0,z_3,0)=q_3(0,0,z_3,0)$ or $\bar{q}_4(0,z_3,0)$ is not zero.
\end{enumerate}

Suppose $M=3$. Then we require either
\begin{enumerate}[(i)]
	\item The rank of the quadratic form form $\bar{q_2}$ is at least $2$,

	\item the rank of $\bar{q_2}$ is $1$ and the following additional condition is satisfied. Without loss of generality we assume that $\bar{q_2}=z_2^2$.
We require that at least one of the polynomials $\bar{q}_3(0,z_3)$, $\bar{q}_4(0,z_3)$, or $\bar{q}_5(0,z_3)$
is not zero,

	\item or the rank of $\bar{q_2}$ is $0$ and the polynomial $\bar{q_3}(1,t)$ has $3$ distinct roots.
\end{enumerate}

The variety $W_X$ satisfies the condition if it satisfies condition at every nonsingular point and $\operatorname{Sing}(W_X)$ is empty or is a unique ordinary double point.

Let $\sigma:X\to\mathds{P}^M$ be the double cover branched over a hypersurface $W_X$.
The variety $X$ satisfies the condition if $W_X$ satisfies the condition.
Note that $X$ is smooth or has a unique ordinary double point, which is the preimage of the ordinary double point on the branching divisor.
\end{Cond}

\begin{Lemma} \label{LemLef}
Let $\sigma:X\to\mathds{P}^M$ be the double cover branched over a hypersurface of degree $2M$.
Suppose $X$ satisfies Condition \ref{Regularity}, then
\begin{align*}
\operatorname{Cl}(X)=\operatorname{Pic}(X)= H\mathds{Z},
\end{align*}
where $H\in \big| -K_X \big|$.
\end{Lemma}
\begin{proof}
First, note that by Hurwitz formula $H$ is a pullback of a hyperplane on $\mathds{P}^M$.
If $X$ is smooth, then the statement of the lemma follows from Lefschetz hyperplane section theorem. 
Suppose $X$ has a unique ordinary double point and suppose $\operatorname{dim}X=3$, then the same is due to \cite{Clemens} (or we could use \cite[Theorem~4.1]{Rams}).
If $\operatorname{dim}X>3$, then \cite[Lemma~28]{Ch08} implies the statement of the lemma.
\end{proof}

The following theorem is a stronger version of \cite[Theorem~4]{Pukh_Sing_Fano}. 

\begin{Th} \label{FiberRig}
Let $\sigma:X\to\mathds{P}^M$ be the double cover branched over a hypersurface of degree $2M$.
Suppose $X$ satisfies Condition \ref{Regularity}.
Then for every effective divisor $D\in \big| -nK_X \big|$ the pair $\big(X,\frac{1}{n}D\big)$ is canonical for any positive $n\in \mathds{Z}$.
\end{Th}
\begin{proof}
Suppose $D$ is a reducible divisor, that is $D=D_1+D_2$. 
By Lemma \ref{LemLef} we may assume $D_i\in \big| -n_iK_X \big|$ for some positive $n_i\in \mathds{Z}$, $i=1,2$.
Clearly $n_1+n_2=n$.
Suppose the theorem holds for $\big(X,\frac{1}{n_i}D_i\big)$, $i=1,2$.
Note that if the pairs $(X,F)$ and $(X,F^\prime)$ are canonical then so is the pair 
\begin{align*}
\big(X,\alpha F +(1-\alpha) F^\prime \big)
\end{align*}
for $0\leqslant \alpha\leqslant 1$.
Set $F=\frac{1}{n_1}D_1$, $F^\prime=\frac{1}{n_2}D_2$, and $\alpha=\frac{n_1}{n}$, then we see that 
\begin{align*}
\alpha F +(1-\alpha) F^\prime=\frac{1}{n}D_1+\frac{1}{n}D_2=\frac{1}{n}D.
\end{align*}
Hence the theorem holds for the pair $\big(X,\frac{1}{n}D\big)$. Thus it is enough to prove the theorem for irreducible divisors $D$.

It follows from \cite[Proof~of~Theorem~4,~page~11]{Pukh_Sing_Fano} that $\big(X,\frac{1}{n}D\big)$ is canonical at the ordinary double point. 
By \cite[p.~23-24]{Pukh_Dir_Prod} the pair $\big(X,\frac{1}{n}D\big)$ may only be singular at the nonsingular points on the ramification divisor.
Let $x \in X \setminus \Sing(X)$ be a point on the ramification divisor.
Let $W_X$ be the branching divisor of the double cover $\sigma: X\to \mathds{P}^M$. 
Let $A$ be the hyperplane in $\mathds{P}^M$ tangent to $W_X$ at the point $\sigma(x)\in W_X$.
Denote $\sigma^{-1}(A)$ as $H$. 
It was shown in \cite[Proof~of~Lemma~5,~p.~29-30]{Pukh_Dir_Prod} that the pair $\big(X,\frac{1}{n}D\big)$ is canonical at $x$, unless $D=H$. In this case $n=1$, since by Hurwitz formula $H\in \big| -K_X \big|$.

We have not used the Condition \ref{Regularity} yet. If $\operatorname{rank}(\bar{q}_2)\geqslant 2$, then by \cite[p.~29-30]{Pukh_Dir_Prod} the pair $(X,H)$ is canonical at $x$. This proves the theorem for $M\geqslant 5$. 

Suppose $M=4$ and $\operatorname{rank}\bar{q}_2=1$, then the local equation of $H$ at $x$ is
\begin{align*}
y^2=z_2^2+c_1z_3^3+c_2z_3^4+z_4(\dots),
\end{align*}
As either $c_1$ or $c_2$ is not zero, $x\in H$ is a singularity of the type $cA_3$ at worst. Thus $H$ has terminal singularities and $(X,H)$ is canonical.

Now suppose $M=3$ and $\bar{q}_2=z_2^2$. It follows from the Condition \ref{Regularity} that the local equation of $H$ is $y^2=z_2^2+z_3^3$, $y^2=z_2^2+z_3^4$, $y^2=z_2^2+z_3^5$, or $y^2=z_2(z_2^2+z_3^2)$ thus $x$ is a Du Val singularity, $H$ has canonical singularities, and therefore the pair $(X,H)$ is canonical.
\end{proof}

It was proven in \cite[Proposition~5]{Pukh_Dir_Prod}, that a generic hypersurface $W_X$ of degree $2M$ satisfies Condition \ref{Regularity}.
Thus a generic fiber of $\pi$ satisfies Condition \ref{Regularity} for a generic $W$. We want every fiber of $\pi$ to satisfy Condition \ref{Regularity}.

\begin{Prop} \label{PropGen}
Let $\mathcal{W}$ be the space of all hypersurfaces of degree $2M$ in $\mathds{P}^M$. Denote the space of all hypersurfaces satisfying Condition \ref{Regularity} by $\mathcal{W}_{reg}\subset\mathcal{W}$. If $M\geqslant3$ then the codimension of $\mathcal{W}\backslash\mathcal{W}_{reg}$ is $2$. 
\end{Prop}
\begin{proof}
It is classically known that the codimension of the locus of hypersurfaces with $2$ or more ordinary double points or with worse singularities is $2$. Thus it is enough to prove that the space of hypersurfaces which do not satisfy Condition \ref{Regularity} at nonsingular points has codimension $\geqslant2$.

Clearly, $\mathcal{W}=H^0\big(\mathds{P}^M,\mathcal{O}(2M)\big)$. Let $V=\mathds{P}^M\times \mathcal{W}$ and let $I$ be the incidence hypersurface in it
\begin{align*}
I=\{(x,Q)\in V \mid Q(x)=0\}.
\end{align*}
Let $p$ and $q$ be the natural projections $p:I\to\mathds{P}^M$ and $q:I\to \mathcal{W}$. Let $Y$ be the subset of ``bad'' pairs, that is
\begin{align*}
Y=\{(x,F)\in I \mid &F~\text{is smooth at}~x~\text{and}~F~\text{does not satisfy Condition \ref{Regularity} at}~x \}.
\end{align*}
To prove the proposition it is enough to show that $\operatorname{codim} q(Y)\geqslant2$. 
To show this it is sufficient to prove that for any $x \in \MP^M$.
$$\operatorname{codim}_{p^{-1}(x)}p^{-1}(x)\cap Y = \operatorname{codim}_I Y \geqslant M+1$$
since the dimension of a fiber of $q$ is $M-1$.

Consider the equation of $F$ in affine coordinates in the neighborhood of a point $x$
\begin{align*}
Q_x=q_1+q_2+\dots+q_{2M}=0,
\end{align*}
where $q_i$ is a homogeneous polynomial of degree $i$. The hypersurface $F$ is smooth at $x$ if and only if $q_1\neq 0$. Thus we may assume that $q_1=z_1$ and take $\bar{q_2}=q_2(0,z_2,\dots,z_M)$. The set of quadratic forms of rank $\leqslant1$ in the variables $z_2,\dots,z_M$ is of codimension 
\begin{align*}
c(M)=\frac{(M-1)(M-2)}{2}.
\end{align*}
When $M\geqslant5$ we have $c(M)\geqslant M+1$. Suppose $M=4$. As the conditions $\bar{q_3}(0,z_3,0)=0$ and $\bar{q_4}(0,z_3,0)=0$ add $2$ to the codimension, the codimension is $c(4)+2=5$. Similarly if $M=3$ the variety of hypersurfaces with $\operatorname{rank}(\bar{q_2})=1$ satisfying the conditions: $\bar{q_3}(0,z_3)=0$, $\bar{q}_4(0,z_3)=0$ and $\bar{q}_5(0,z_3)=0$, is of codimension $4$. On the other hand, the variety of hypersurfaces with $\bar{q_2}=0$ is of codimension $3$ and the condition on $\bar{q_3}(1,t)=0$ having a multiple root, adds $1$ to the codimension.
\end{proof}

\begin{Lemma} \label{KBig4}
Suppose $V$ is the variety from Theorem \ref{TechTh}, $(i)$, then $-K_V$ is not big.
\end{Lemma}
\begin{proof}
It is easy to compute that 
\begin{align*}
-K_V=g^*(L),
\end{align*}
where $L$ is a divisor of bi-degree $(1,2-l)$, $l\geqslant3$. Thus $-K_V$ is not big.
\end{proof}

\begin{proof}[Proof~of~Theorem~1.1 (i)]
Let $\pi:V\to \mathds{P}^1$ be the composition of the double cover and the projection onto $\mathds{P}^1$. There is a map $\bar{\pi}: \mathds{P}^1\to \mathcal{W}$ corresponding to the fibration $\pi$: $\bar{\pi}$ maps $t\in \mathds{P}$ to the branching divisor of the double cover $F_t \to \mathds{P}^M$, where $F_t$ is the fiber over $t$. The image $\bar{\pi}(\mathds{P}^1)$ is a curve of degree $2l$ since $W$ is a divisor of bi-degree $(2M,2l)$. By Proposition \ref{PropGen} the codimension of the set of hypersurfaces which do not satisfy Condition \ref{Regularity} is $2$, therefore a generic rational curve of degree $2l$ does not intersect this set. Thus for a generic divisor $W$ every fiber of $\pi$ satisfies Condition \ref{Regularity}.

Let $\chi$ be a birational map to a variety admitting a Mori fiber space. Suppose $\chi$ is not an isomorphism. Proposition 2.8 is applicable since $-K_V$ is not big by Lemma \ref{KBig4}. 
Thus there is a linear system $\mathcal{M}$ and rational numbers $\lambda$, $s$ such that $K_V+\lambda\mathcal{M}\sim_\mathds{Q} sF$ and $(V,\lambda\mathcal{M})$ is not canonical. Suppose the pair is not canonical at $Z$. Suppose $Z$ is in a fiber $F$ of $\pi$. Then $(F,\lambda\mathcal{M}\big\vert_F)$ is not log canonical at $Z$ by Remark \ref{Adj}. Since $\lambda\mathcal{M}\big\vert_F \sim_\mathds{Q} -K_F$ it contradicts Theorem \ref{FiberRig}. Thus $Z$ is not in any fiber of $\pi$. Consider a generic fiber $F$ of $\pi$, then $(F,\lambda \mathcal{M}\big\vert_F)$ is not canonical at $Z\cap F$. This also contradicts Theorem \ref{FiberRig}, thus $\chi$ is an isomorphism. 
\end{proof}

%%%%%%%%%%%%%%%%%%%%%%%%%%%%%%%%%%%%%%%%%%%%%%%%%%%%%%%%%%%
%%%%%%%%%%%%%%%%%%%%%%%%%%%%%%%%%%%%%%%%%%%%%%%%%%%%%%%%%%%
%%%%%%%%%%%%%%%%%    Dimension 3     %%%%%%%%%%%%%%%%%%%%%%
%%%%%%%%%%%%%%%%%%%%%%%%%%%%%%%%%%%%%%%%%%%%%%%%%%%%%%%%%%%
%%%%%%%%%%%%%%%%%%%%%%%%%%%%%%%%%%%%%%%%%%%%%%%%%%%%%%%%%%%

\section {Example in dimension 3}
In this section we construct an example in dimension $3$ and prove Theorem \ref{TechTh}, $(ii)$.

Let $Y=(\mathds{C}^6\setminus Z\big(\langle u,v\rangle\cap\langle x,y,z,w\rangle\big)/(\mathds{C}^*)^2$, where the $(\mathds{C}^*)^2$-action is given by the matrix 

\[\left(\begin{array}{cccccc}
u&v&x&y&z&w\\
0&0&1&1&2&3\\
1&1&-3&-3&0&0\end{array}\right)\]
and let 
\begin{align*}
Q=w^2+z^3+(u^{12}+v^{12})M_4(x,y)z+R_{18}(u,v)x^2y^2(x-y)^2,
\end{align*}
where $M_4$ and  $R_{18}$ are homogeneous polynomials of degrees $4$ and $18$ respectively. Let $\mathcal{L}$ be the set of hypersurfaces given by $Q=0$ for all the different $M_4$ and $R_{18}$. Clearly $\mathcal{L}$ is a linear system of divisors. Let $X$ be a generic variety in $\mathcal{L}$. The variety $Y$ is a projective simplicial toric variety with the Cox ring $\mathds{C}[u,v,x,y,z,w]$. This ring has 2 gradings given by the matrix above. Divisors in $Y$ are the zero sets of homogeneous polynomials in the Cox ring. The grading of the Cox ring defines the degree of divisors on $X$, for example
\begin{align*}
\deg X = \deg Q = (6,0).
\end{align*}

Let $\pi_Y:Y\to\mathds{P}^1$ be defined by 
\begin{align*}
(u:v:x:y:z:w)\mapsto(u:v),
\end{align*}
clearly, this map is a $\mathds{P}(1,1,2,3)$-fibration. Let $\pi=\pi_Y\vert_X$, then a generic fiber of $\pi$ is del Pezzo surface of degree $1$. Indeed, for a fixed $(u:v)$ the fiber is a hypersurface of degree $6$ in $\mathds{P}(1,1,2,3)$ which is a del Pezzo surface of degree $1$. 

Denote the torus-invariant divisor given by $\alpha = 0$, $\alpha \in \Theta = \{u,v,x,y,z,w\}$, as $D_\alpha$, and note that $D_u\sim D_v \sim F$, where $F$ is a fiber of $\pi_Y$. It is easy to see that $\operatorname{Cl}(Y)=D_y \mathds{Z}\oplus F\mathds{Z}$.

\begin{Lemma} \label{LemRams}
Let $F$ be a fiber of $\pi_Y:Y\to\mathds{P}^1$ and suppose $X\in\mathcal{L}$. Then
\begin{enumerate}[(i)]
	\item $\big|X\big|=\big|6D_y+18F\big|$ is base point free,
	\item $X+F$ is ample,
	\item $X$ is big.
\end{enumerate}
\end{Lemma}
\begin{proof}
The equation of $D\in\big|X\big|$ may contain monomials: $w^2$, $z^3$, $x^6u^{18}$, $x^6v^{18}$, $y^6u^{18}$, and $y^6v^{18}$ which are not all equal zero at the same time at any point on $Y$, thus $(i)$ holds.

Suppose $C$ is a curve in a fiber then
\begin{align*}
C\cdot (X+F)=C\cdot X =C\cdot X\big\vert_F=\deg \big( \mathcal{O}_{\mathds{P}(1,1,2,3)}(6)\big)\big\vert_C>0.
\end{align*} 
Suppose a curve $C$ is not in a fiber. Then since $\big| X \big|$ is base point free 
\begin{align*}
C\cdot (X+F)\geqslant C\cdot F>0.
\end{align*}
Thus $X+F$ is ample by Kleiman criterion.

Clearly $2X\sim (X+F)+(6D_y+17F)$, hence $(iii)$ follows from $(ii)$.
\end{proof}

\begin{Lemma} \label{SingX}
Let $X$ be a generic divisor in a linear system $\mathcal{L}$.
Then the following assertions hold:
\begin{enumerate}[(i)]
	\item There are $108$ cuspidal curves of anticanonical degree $1$ in fibers of $\pi$: $72$ of them are given by the equations $R_{18}=M_4=0$, the other $36$ curves are given by $u^{12}+v^{12}=0$ and $x=0$, $y=0$, or $x=y$.

	\item Let $F$ be a fiber of $\pi$ and let $C\subset F$ be one of the $72$ curves. Let $P$ be the cusp of $C$, then $P$ is an ordinary double point of $F$.

	\item Let $C$ be one of the $36$ curves and let $P$ be the cusp of $C$. Then $P$ is an ordinary double point of $X$.

	\item The variety $X$ is smooth outside of the $36$ ordinary double points described in $(iii)$.
\end{enumerate}
\end{Lemma}
\begin{proof}
A fiber of $\pi$ is defined by the ratio $(u:v)$ and curves of degree $1$ in it by $(x:y)$. 
Once these ratios are fixed the curve is given by the equation
\begin{align*}
w^2 + z^3 + a_1 s^4 z + a_0 s^6 = 0
\end{align*}
in a weighted projective space $\MP(1_s, 2_z, 3_w)$.
Note that the coefficient at $z^2$ is zero for every fiber, therefore the curve is cuspidal if and only if $a_1 = a_0 = 0$, that is if and only if
\begin{align*}
R_{18}(u,v)x^2&y^2(x-y)^2=0,\\
(u^{12}+v^{12}&)M_4(x,y)=0.
\end{align*}
As $R_{18}$ and $M_4$ are generic we must have $R_{18}=0$ and $M_4=0$ ($72$ curves) or $u^{12}+v^{12}=0$ and $x^2y^2(x-y)^2=0$ ($36$ curves), thus $(i)$ holds. Note that the cusps of these curves are at $w=z=0$.

The local equation of $X$ at the cusp of one of the $72$ curves is
\begin{align*}
w^2+z^3+zs(x,y)+t(u,v)=0,
\end{align*}
where $t$ is a linear factor of $R_{18}$ and $s$ is a linear factor of $M_4$. Clearly the fiber $t=0$ has an ordinary double point at $w=z=s(x,y)=0$.

If $C$ is one of the $36$ curves, then the local equation of $X$ at the cusp of $C$ is
\begin{align*}
w^2+z^3+zt(u,v)+s^2(x,y)=0,
\end{align*}
where $s=x$, $s=y$, or $s=x-y$ and $t$ is a linear factor of $u^{12}+v^{12}$. Clearly $X$ has an ordinary double point at $w=z=t(u,v)=s(x,y)=0$.

Note that $X$ does not pass through the singular locus of $Y$. Suppose $X$ is singular at the point $P$ with coordinates $(x,y,z,w,u,v)$. Then, clearly $w=0$. By Bertini's theorem \cite[Theorem~4.1]{KSP} the point $P$ is a base point of $\mathcal{L}$.

There is a polynomial $t\in\mathds{C}[u,v]$ of degree $1$ such that $t\neq0$ at $P$. Let $X_t\in \mathcal{L}$ be a variety with $R_{18}=t^{18}$. Suppose $xy(x-y)\neq 0$ then for some $c\in\mathds{C}$ the point $P$ does not lie on $X_{ct}$. Thus $x=0$, $y=0$, or $x=y$.

Suppose $z(u^{12}+v^{12})\neq0$ at $P$. Then repeating the argument for some polynomial $s\in \mathds{C}[x,y]$ of degree $1$ and $M_4=(cs)^4$ we conclude that $P$ is not a base point of $\mathcal{L}$. Thus $z=0$ or $(u^{12}+v^{12})=0$. Since the polynomial $M_4$ is generic, we may assume that $M_4\neq 0$ at $P$. Therefore $(u^{12}+v^{12})=0$ if and only if $z=0$, because
$$\frac{\partial Q_P}{\partial z}=3z^2+(u^{12}+v^{12})M_4(x,y)=0.$$

Thus every singular point must satisfy $w=z=u^{12}+v^{12}=xy(x-y)=0$. There are $36$ points satisfying these equations and they are described in $(iii)$. Note that there are $3$ ordinary double points in the fiber $t(u,v)=0$, where $t$ is a linear factor of $u^{12}+v^{12}$.
\end{proof}

We now prove that the fibration $\pi:X\to \mathds{P}^1$ is a Mori fiber space. We already know that $X$ is terminal and that a generic fiber is a del Pezzo surface. 
Thus we only need to show that it is $\mathds{Q}$-factorial and that the relative Picard rank $\rho(X/\mathds{P}^1)=1$.

Let $U=Y\setminus \operatorname{Sing}(Y)$ and let $i:U\to Y$ be the inclusion map, then denote $\bar{\Omega}_Y^1=i_*(\Omega_U^1)$.

\begin{Th}[{\cite[Theorem~4.1]{Rams}}] \label{RamsTh}
Let $Y$ be a Cohen-Macaulay fourfold and let $X\subset Y$ be a hypersurface such that $\operatorname{Sing}(Y)\cap X=\emptyset$ and $\operatorname{Sing}(X)$ consists only of ordinary double points. Let $\widetilde{X}$ be the resolution of $X$ obtained by blowing up the ordinary double points. Let $\mu_X$ be the number of singular points on $X$ and let $\mathcal{I}_{\operatorname{Sing}(X)}$ be the sheaf of ideals on $Y$ of singular points of $X$. 
Assume that the following conditions are satisfied:
\begin{enumerate}[(i)]

	\item $H^i\big(Y,\mathcal{O}_Y(-2X)\big)=0$ for $i=1,2,3$;

	\item $H^i\big(Y,\mathcal{O}_Y(-X)\big)=0$ for $i=1,2,3$;

	\item $H^i\big(Y,\bar{\Omega}_Y^1\otimes\mathcal{O}_Y(-X)\big)=0$ for $i=1,2,3$;

	\item $H^2(Y,\bar{\Omega}_Y^1)=0$.
\end{enumerate}
Denote $\delta_X=h^0\big(Y,\mathcal{O}(K_Y+2X)\otimes\mathcal{I}_{\operatorname{Sing}(X)}\big)-\big(h^0(Y,\mathcal{O}(K_Y+2X))-\mu_X\big)$. 
Then 
$$h^{1,1}(\widetilde{X})=h^1(Y,\bar{\Omega}^1_Y)+\mu_X+\delta_X.$$
\end{Th}

The number $h^0\big(Y,\mathcal{O}(K_Y+2X)\otimes\mathcal{I}_{\operatorname{Sing}(X)}\big)$ is the dimension of the space of the divisors $D$ such that $D \sim K_Y + 2 X$ and $\Sing (X) \subset \Supp (D)$.
The number $h^0(Y,\mathcal{O}(K_Y+2X))-\mu_X$ is the expected dimension of that space, that is the dimension if the singularities of $X$ are in a general position.
Thus the number $\delta_X$ is the difference between the expected and the actual dimension of the space and is always non-negative. 
It is known as the \emph{defect} of a hypersurface \cite{Rams}. 

\begin{Lemma}
For a generic $X\in\mathcal{L}$ the defect $\delta_X=0$.
\end{Lemma}
\begin{proof}
Since $-K_Y$ is the sum of the torus-invariant divisors we have
\begin{align*}
\operatorname{deg} K_Y=(-1-1-2-3,-1-1+3+3)=(-7,4),
\end{align*}
hence $K_Y\sim_\mathds{Q} -7D_y-17F$ as $\operatorname{deg} D_y=(1,-3)$. Let $P_i$, $i=1,\dots, 36$, be the ordinary double points of $X$. It is enough show that there are divisors 
\begin{align*}
H_i\in\big|2X+K_Y\big|=\big|5D_x+19F\big|, \quad i=1\dots 36,
\end{align*}
such that $P_j\in H_i$ for any $j\neq i$ and $P_i\not\in H_i$. Let $F_k$ be the fibers containing the singular points of $X$, there are $12$ of them, since there are $3$ ordinary double points in each $F_k$. Let $D_{x-y}$ be the divisor defined by $y=x$. Taking $11$ fibers $F_k$ and $2$ divisors out of $D_x$, $D_y$, $D_{x-y}$ we get a divisor $H^\prime_i$ linearly equivalent to $2D_x+11F$ passing through any $35$ singular points out of $36$. The divisor $A$ given by the equation $(x-2y)^3u^{8}$ does not pass through the singular points of $X$, therefore we may set $H_i=H^\prime_i+A$.
\end{proof}

\begin{Prop}
Suppose $X\in\mathcal{L}$ is generic. Then $X$ is $\mathds{Q}$-factorial and 
$$\operatorname{Pic}X\otimes\mathds{Q}=\big(D_y\big\vert_X \big) \mathds{Q} \oplus F \mathds{Q},$$
where $F$ is a fiber of $\pi$.

\end{Prop}
\begin{proof}
Let us show that $X$ and $Y$ satisfy the conditions of Theorem \ref{RamsTh}.
By Lemma \ref{LemRams} the divisor $X$ is big and nef, therefore the toric Kawamata-Viehweg vanishing theorem (\cite[Theorem~9.3.10]{CLS}) implies $H^i\big(Y,\mathcal{O}_Y(K_Y+2X)\big)=0$, $i>0$. Thus by the toric Serre Duality theorem (\cite[Theorem~9.2.10]{CLS}) we have $(i)$. Using the same argument we get $(ii)$.
 
As $\operatorname{Cl}(Y)\cong\mathds{Z}\oplus\mathds{Z}$ the following sequence is exact by \cite[Theorem~8.1.6]{CLS}.

\begin{align*}
0\to \bar{\Omega}_Y^1(-X)\to\bigoplus_{\alpha \in \Theta}\mathcal{O}_Y(-D_\alpha-X)\to \mathcal{O}_Y(-X)\oplus \mathcal{O}_Y(-X)\to 0.
\end{align*}
By taking associated long exact sequence of cohomologies and applying $(ii)$ we get
\begin{align*}
0\to H^i\big(Y,\bar{\Omega}_Y^1(-X)\big)\to \bigoplus_{\alpha \in \Theta} H^i\big(Y,\mathcal{O}_Y(-D_\alpha-X)\big), \quad i=1,2,3.
\end{align*}
It is enough to prove that $H^i\big(Y,\mathcal{O}_Y(-D_\alpha-X)\big)=0$ for all $s\in Z$, $i=1,2,3$. If $s\in{u,v,z,w}$ then $D_\alpha+X$ is big and nef, therefore by Serre duality and Kawamata-Viehweg vanishing $H^i\big(Y,\mathcal{O}_Y(-D_\alpha-X)\big)=0$. Let $F$ be a fiber of $\pi$. Then $\mathcal{O}_F(F)\cong \mathcal{O}_F$ and the sequence
\begin{align*}
0\to\mathcal{O}_Y\to\mathcal{O}_Y(F)\to\mathcal{O}_F\to0
\end{align*}
is exact. Hence the sequences
\begin{align*}
H^i(Y,\mathcal{O}_Y)\to H^i\big(Y,\mathcal{O}_Y(F)\big)\to H^i(F,\mathcal{O}_F), \quad i=1,2,3
\end{align*}
are exact. Applying \cite[Theorem~9.3.2]{CLS} we see that
\begin{align*}
H^i(Y,\mathcal{O}_Y)=H^i(F,\mathcal{O}_F)=0, \quad i=1,2,3,
\end{align*}
hence $H^i\big(Y,\mathcal{O}_Y(F)\big)=0$. As $-X-D_y\sim_\mathds{Q}-X-D_y\sim_\mathds{Q} K_Y-F$ by toric Serre duality 
\begin{align*}
H^i\big(Y,\mathcal{O}_Y(-D_x-X)\big)\cong H^{4-i}\big(Y,\mathcal{O}_Y(F)\big)=0.
\end{align*}
The condition $H^2(Y,\bar{\Omega}_Y^1)=0$ is satisfied because $Y$ is a complete toric simplicial variety. 

Let $\widetilde{X}$ be the resolution of $X$ acquired by blowing up the $36$ double points. We have shown $X$ and $Y$ satisfy the requirements of Theorem \ref{RamsTh} therefore $\operatorname{rk}\operatorname{Pic}(\widetilde{X})=38$. On the other hand $\operatorname{rk}\operatorname{Pic}(\widetilde{X})\geqslant \operatorname{rk}\operatorname{Cl}(X)+36$, thus $\operatorname{rk}\operatorname{Pic}(X)=\operatorname{rk}\operatorname{Cl}(X)=2$ and $X$ is $\mathds{Q}$-factorial. Clearly, the divisors $D_y\big\vert_X$ and $F\big\vert_X$ are generators.
\end{proof}

\begin{Lemma} \label{KBig3}
The anticanonical class of the variety $X$ is not big.
\end{Lemma}
\begin{proof}
We have $K_Y=-7D_y-17F$, therefore by adjunction the anticanonical class of $X$ is $-K_X=D_y\big\vert_X-F\big\vert_X$. Thus $\big|-nK_X\big|=\emptyset$ for any $n>0$.
\end{proof}

\begin{Lemma}[{\cite[Proposition~3.2]{JPDP}}]
Let $F$ be a del Pezzo surface of degree $1$ and let $C\in\big|-K_F\big|$ be an irreducible curve. Suppose $(F,C)$ is not log canonical at $P$ and $F$ is smooth at $P$. Then $C$ is a cuspidal rational curve. 
\end{Lemma}

%\begin{Cor}
%Let $F$ be a del Pezzo surface of degree $1$ with at worst ordinary double points and let $D\sim-K_F$ be a $\mathds{Q}$-Cartier divisor. If $(F,D)$ is not log canonical at $P$, then there is a rational curve $C\in\big|-K_F\big|$ such that $P\in C$ is a cusp.
%\end{Cor}

\begin{Lemma}[{\cite[proof of Theorem~3.10]{Corti}}] \label{LemCorti}
Let $X$ be a $3$-dimensional variety and let $D$ be an effective $\mathds{Q}$-Cartier divisor on $X$. Suppose $P\in X$ is an ordinary double point, let $f:X^+\to X$ be the blow up of $P$, and let $E$ be the exceptional divisor. Let $\mathcal{M}$ be a mobile linear system on $X$, then the pair $(X,\lambda \mathcal{M})$ is not canonical at the point $P$ if an only if $a(E,X,D)<0$.
\end{Lemma}

The theorem we refer to ({\cite[Theorem~3.10]{Corti}}) states that the only extremal contraction to the ordinary double point in the category of terminal $\mathds{Q}$-factorial varieties is the blow up of the ordinary double point. To show this Corti proves that $a(F,X,\lambda\mathcal{M})$ is minimal if and only if $F=E$, which implies Lemma \ref{LemCorti}.

\begin{proof}[Proof~of~Theorem~1.1~(ii)]
Let $\chi$ be a birational map to a variety admitting a Mori fiber space and suppose $\chi$ is not an isomorphism. Let $F$ be a fiber of $\pi$. Then by Proposition 2.8 there is a linear system $\mathcal{M}$ and numbers $\lambda,s\in \mathds{Q}_{>0}$ such that $K_X+\lambda\mathcal{M}\sim_\mathds{Q} sF$ and $(X,\lambda \mathcal{M})$ is not canonical. 

Suppose the pair is not canonical at a curve $C$, then $\operatorname{mult}_C \lambda \mathcal{M}>1$. 
Suppose $C \subset F$ for some fiber $F$ and let $D$ be a generic divisor in $\mathcal{M}$.
Then we have $\lambda D \vert_F = C + C^\prime$, where $C^\prime$ is effective.
Thus 
$$D \cdot F \cdot (-K_X) > C \cdot (-K_X)$$
 since $-K_X$ is $\pi$-ample.
Since $\Sing (X)$ consists of only ordinary double points, the variety $X$ is Gorenstein, in particular $K_X$ is a Cartier divisor.
Thus $C\cdot (-K_X) = k \geqslant 1$, $k \in \MZ$.
On the other hand 
$$\lambda D \cdot F \cdot (-K_X) = K_X^2 \cdot F = 1$$
 since $\pi: X \to \MP^1$ is a del Pezzo fibration of degree $1$.
Combining it all together we get
\begin{align*}
1=\lambda D \cdot F \cdot (-K_X) > C \cdot (-K_X)\geqslant 1,
\end{align*}
contradiction.

Suppose $C$ is not in any fiber, then for a generic fiber $F$ let $P\in F\cap C$. 
Since $F$ is generic and $\mathcal{M}$ is mobile, the system $\mathcal{M}\big\vert_F$ does not have fixed components. Clearly $\operatorname{mult}_P \lambda\mathcal{M}\big\vert_F >1$, but for generic curves $Z_1, Z_2\in \mathcal{M}\big\vert_F$ we have

\begin{align*}
1=\lambda^2 Z_1\cdot Z_2\geqslant \lambda^2(Z_1\cdot Z_2)_P\geqslant \big( \operatorname{mult}_P \lambda\mathcal{M}\big\vert_F \big)^2 >1,
\end{align*}
where the first equality holds since $F$ is a del Pezzo surface of degree $1$.

Thus the pair is not canonical at some point $P$. Let $F$ be the fiber containing $P$, let $f:X^+\to X$ be the blow up of $P$, let $E$ be the exceptional divisor of $f$, let $D^+$ be the proper transform of $D$, and let $a$ be the number defined by the equality 
$f^*(\lambda D)=\lambda D^+ + a E$. Let $C$ be a generic curve in $\big| -2K_F \big|$ passing through $P$ and let $C^+$ be its proper transform on $X^+$. Note that $\lambda D\cdot C=-2K_F^2=2$.

Suppose $P$ is an ordinary double point of $X$. Then 
\begin{align*}
K_{X^+}+\lambda D^+(a-1)E \sim f^*(K_X+\lambda D),
\end{align*}
and $a>1$ by Lemma \ref{LemCorti}. The surface $F$ is singular at $P$ since $F$ is a Cartier divisor on $X$. The inequality $C^+\cdot E\geqslant 2$ holds since $F$ is singular at $P$ and $C$ is a Cartier divisor on $F$, therefore
\begin{align*}
0\leqslant \lambda D^+\cdot C^+=f^*(\lambda D)\cdot C^+ - a C^+\cdot E = \lambda D\cdot C-2a = 2-2a<0
\end{align*}
contradiction.

Thus $X$ is smooth at $P$. Then $a=\operatorname{mult}_P \lambda D>1$ since the pair $(X,\lambda D)$ is not canonical at $P$. Suppose $P$ is a singular point of $F$. The inequality $C^+\cdot E\geqslant 2$ holds again, therefore
\begin{align*}
0\leqslant \lambda D^+\cdot C^+=f^*(\lambda D)\cdot C^+ - a C^+\cdot E \leqslant \lambda D\cdot C-2a = 2-2a<0,
\end{align*}
contradiction.

Thus we may assume that $F$ is also smooth at $P$. Let $D$ be a generic divisor in a linear system $\mathcal{M}$. Then the pair $(F, \lambda D\big\vert_F)$ is not log canonical at $P$ by Remark \ref{Adj}. Let $C\in \big| -K_F \big|$ be the curve passing through $P$. By construction $C$ is smooth or nodal, therefore $(F,C)$ is log canonical. Let $A\in \big| -nK_F \big|$ such that $\operatorname{Supp}(A)$ does not contain $C$. Then 
\begin{align*}
\operatorname{mult}_P A \leqslant A\cdot C = n,
\end{align*}
hence the pair $(F,\frac{1}{n}A)$ is log canonical at $P$. Therefore $\big(F,\alpha C + \frac{1-\alpha}{n}A\big)$ is log canonical at $P$ for any $0\leqslant \alpha \leqslant 1$. In particular $(F,\lambda D\big\vert_F)$ is log canonical, contradiction.
\end{proof}

\begin{Remark}
It was shown in \cite[Corollary~7.5]{CSh11} that a del Pezzo fibration $\pi: X\to \mathds{P}^1$ of degree $1$ is birationally rigid if for every del Pezzo fibration $\pi_Y: Y\to \mathds{P}^1$, such that there is a fiberpreserving birational map $\chi: X \dasharrow Y$, the variety $Y$ satisfies the $K$-condition. It is unclear how to check the requirements of the corollary and there are no known examples of varieties satisfying this property. The variety $X$ from Theorem \ref{TechTh} satisfies this property trivially, however: $\pi:X\to \mathds{P}^1$ is a unique del Pezzo fibration in the birational class and $X$ satisfies the $K$-condition.
\end{Remark}

%%%%%%%%%%%%%%%%%%%%%%%%%%%%%%%%%%%%%%%%%%%%%%%%%%%%%%%%%%%
%%%%%%%%%%%%%%%%%%%%%%%%%%%%%%%%%%%%%%%%%%%%%%%%%%%%%%%%%%%
%%%%%%%%%%%%%%%%%      Epilogue      %%%%%%%%%%%%%%%%%%%%%%
%%%%%%%%%%%%%%%%%%%%%%%%%%%%%%%%%%%%%%%%%%%%%%%%%%%%%%%%%%%
%%%%%%%%%%%%%%%%%%%%%%%%%%%%%%%%%%%%%%%%%%%%%%%%%%%%%%%%%%%

\section{Rationally connected non-Fano type varieties}
In this section we prove Theorem \ref{MainTh} and discuss generalizations.

\begin{Prop} \label{MMP}
Let $X$ be a variety of Fano type. Then there is a variety $V$ admitting a Mori fiber space such that $V$ is birational to $X$ and $-K_V$ is big.
\end{Prop}
\begin{proof}
There is an effective divisor $D$ on $X$ such that $-(K_X+D)$ is ample and $(X,D)$ is klt. Since $-K_X$ is a sum of an effective and an ample divisor it is big. If $X$ is not $\mathds{Q}$-factorial we may replace it by its $\mathds{Q}$-factorization $Y$ (\cite[Corollary~1.4.3]{BCHM}). Note that there is a morphism $g:Y\to X$ which is an isomorphism in codimension $1$, therefore  $-K_Y$ is big and $(Y,D_Y)$ is klt, where $D_Y$ is the proper transform of $D$ on $Y$. If $X$ is $\mathds{Q}$-factorial, set $Y=X$. Since $-K_Y$ is $\mathds{Q}$-Cartier and $(Y,D_Y)$ is klt, the variety $Y$ has log terminal singularities. There is a terminalization morphism $f:Z\to Y$, that is a birational morphism such that $Z$ has terminal singularities and all exceptional divisors $E_i$ of $f$ satisfy $a(E_i,Y)\leqslant 0$ (\cite[Corollary 1.4.3]{BCHM}). The anticanonical class of $Z$ is big since
\begin{align*}
-K_Z=-f^*K_Y-\sum a(E_i,Y)E_i.
\end{align*}
Suppose $V$ is a result of running the MMP on $Z$. We claim that $-K_V$ is big. Indeed, for a divisorial contraction $h:W\to U$ with the exceptional divisor $E$ we can write
\begin{align*}
-h^*K_U=-K_W + a E, \quad a>0,
\end{align*}
hence $-K_U$ is big if $-K_W$ is big. Isomorphisms in codimension $1$ preserve the property of divisors being big, therefore the anticanonical class is big on every step of the MMP, in particular $-K_V$ is big. The variety $X$ is rationally connected by \cite[Theorem~1]{Zh} therefore $V$ admits a Mori fiber space.
\end{proof}

\begin{proof}[Proof~of~Theorem~1.2]
Suppose $X$ (or $V$) is birational to a variety of Fano type. Then by Proposition \ref{MMP} the variety $X$ (resp. $V$) is birational to a variety $Y$ admitting a Mori fiber space and satisfying $-K_Y$ is big. Theorem \ref{TechTh} implies $Y\cong X$ (resp. $Y\cong V$) but $-K_X$ (resp. $-K_V$) is not big by Lemma \ref{KBig3} (Lemma \ref{KBig4}), contradiction. 
\end{proof}

\begin{Lemma} \label{RC}
The varieties $V$ and $X$ described in Theorem \ref{TechTh} are rationally connected.
\end{Lemma}
\begin{proof}
Since $V$ and $X$ are Fano fibrations over $\mathds{P}^1$ they are rationally connected by \cite[Theorem~0.1]{KMM} and \cite[Corollary~1.3]{GHJ}.
\end{proof}

Thus the varieties $V$ and $X$ are examples of rationally connected varieties which are not birational to varieties of Fano type.

\begin{Remark}
We could also construct an example by using fibrations onto Fano hypersurfaces of index one. But it is more tiresome and provides examples only for dimension $\geqslant 9$.
\end{Remark}

\begin{Remark}
While this paper have been under review other examples in dimension three have been constructed in \cite{K16}. The example there are conic bundles with sufficiently big discriminant curve. More precisely, Kollár gives examples of rationally connected threefolds which are not birational to Calabi-Yau pairs.
\end{Remark}

\begin{Remark} 
We can run $D$-MMP on a Mori dream space for any divisor $D$ which is not nef \cite[Proposition~1.11]{HK}. Thus, we have another class of varieties which behave very well under the $D$-MMP. It has been proven in \cite[Corollary~1.3.1]{BCHM}, that every $\mathds{Q}$-factorial variety of Fano type is a Mori dream space. The converse is not true even for smooth Mori dream spaces: there exists a smooth rational Mori dream space of dimension $2$ which is not of Fano type \cite[Section~3]{BRS}. In fact, a $\mathds{Q}$-factorial normal projective variety is of Fano type if and only if it is a Mori dream space and spectrum of its Cox ring has only log terminal singularities \cite[Theorem~1.1]{GOST}. One could ask Question \ref{Q1} for Mori dream spaces instead of varieties of Fano type.
We expect that our examples in dimension $\geqslant 4$ are Mori dream spaces but we do not know anything about dimension $3$.
\end{Remark}

\end{document}